\documentclass[12pt]{amsart}
\usepackage{amsbsy,amssymb,amscd,amsmath,amsthm, hyperref}

\newcommand{\issue}[1]{\ no.#1,\ }
\newcommand{\publ}{\ }

\newcommand{\inbook}{\ }
\newcommand{\out}[1]{\ }

\newcommand{\Cal}{\mathcal}
\let\cal\mathcal

\newcommand{\DD}{{\mathbb D}}

\newcommand{\calf}{{\mathcal F}}

\newcommand{\QQ}{{\mathbb Q}}
\newcommand{\RR}{{\mathbb R}}
\newcommand{\CC}{{\mathbb C}}
\newcommand{\ZZ}{{\mathbb Z}}

\renewcommand{\ge}{\geqslant}
\renewcommand{\le}{\leqslant}
\renewcommand{\phi}{\varphi}

\let\Bbb\mathbb
\let\comp\complement
\newcommand{\eps}{\varepsilon}

\DeclareMathOperator{\PSH}{PSH}
\DeclareMathOperator{\SH}{SH}
\DeclareMathOperator{\fSH}{f-SH}
\DeclareMathOperator{\flimsup}{f-\limsup}
\DeclareMathOperator{\FPSH}{\calf-PSH}
\DeclareMathOperator{\Capa}{Cap}
\newtheorem{theorem}{Theorem}[section]
\newtheorem{corollary}[theorem]{Corollary}
\newtheorem{lemma}[theorem]{Lemma}
\newtheorem{prop}[theorem]{Proposition}
\newtheorem{conjecture}[theorem]{Conjecture}

\theoremstyle{definition}
\newtheorem{definition}[theorem]{Definition}

\theoremstyle{remark}
\newtheorem{remark}[theorem]{Remark}

\numberwithin{equation}{section}

\begin{document}

\title{Plurifine Potential theory}

\dedicatory{Dedicated to professor J\'ozef Siciak on the occasion of his 80'th birthday.}
\author{Jan Wiegerinck}
\address{KdV Institute for Mathematics
\\University of Amsterdam
\\Science Park 904
\\P.O. box 94248, 1090 GE Amsterdam
\\The Netherlands}
\email{j.j.o.o.wiegerinck@uva.nl}

\begin{abstract} We give an overview of the recent developments in plurifine pluripotential theory, i.e. the theory of plurifinely plurisubharmonic functions.
\end{abstract}

\maketitle

\section{Introduction}\label{sec0}
The {\it pluri\-fine topology} $\calf$ on (open subsets of) $\CC^n$ was
 introduced by Fuglede in \cite{F6} as the weakest topology in which all
pluri\-sub\-har\-monic functions are con\-tinu\-ous. The concept is analogous to
the H.\ Cartan fine topology on $\Bbb R^n$, in which all subharmonic functions are continuous.  In particular, on $\CC\cong
\RR^2$ fine and plurifine topology coincide. This topology was next considered by Bedford and Taylor \cite{BT2} and they employed it to make precise statements about the convergence of the Monge-Amp\`ere masses associated to sequences of plurisubharmonic functions. The plurifine topology was  further developed in \cite{EW1}, \cite{EW2}. Notions pertaining to the plurifine topology are indicated with the prefix $\calf$ and notions pertaining to the fine topology are indicated with $f$; e.g.,~$f$-open is open in the fine topology. The fine boundary of a set $V$ is denoted by $\partial_f V$.

In one complex variable there is a good theory of finely subharmonic, respectively finely holomorphic, functions, which was largely developed by Fuglede, cf.~\cite{F1,F2,F3}. Once the topology $\calf$ is available, it is natural to try and define \emph{plurifinely plurisubharmonic functions} and \emph{plurifinely holomorphic functions}. It turns out, cf.~  \cite{ElKa,EW2,EW3,EFW11}, that there are two reasonable ways of extending the concept of finely subharmonic, respectively finely holomorphic, functions to several complex variables. A weak concept, defined by demanding that restrictions to complex lines are finely subharmonic, respectively holomorphic, and a strong concept based on approximation by ordinary plurisubharmonic, respectively holomorphic, functions.

With these plurifinely plurisubharmonic functions come several questions: which properties of ordinary plurisubharmonic functions transfer to the new concepts, are the weak and strong concept the same, can we get a hold on the structure of plurifinely plurisubharmonic functions? In short: develop plurifine potential theory! In this overview we will discuss what has been achieved so far. We will also describe how results on pluripolar hulls, obtained e.g. by Siciak \cite{Si03}, Zwonek, \cite{Zw05}, Edlund and J\"oricke, \cite{EJ}, and Edigarian, El Marzguioui and the author,  \cite{ EEW, EdW1, EW03, EW04} are clarified with the help of plurifinely plurisubharmonic functions and fine analytic continuation.  In Section \ref{sec3} we will see that certain questions about the structure of pluripolar hulls and complete pluripolar sets are more naturally formulated and partially solved by means of plurifine potential theory.

In general we will refrain from giving proofs and refer to the literature instead. A few exceptions are made, in particular for the results in subsection \ref{Bounded} and \ref{S-polar} that are not in the literature yet. These sections were spurred by pertinent questions by Urban Cegrell and Peter Pflug.

\subsection*{Acknowledgements}
It is with pleasure that I thank the referees for their useful comments, and Iris Smit for carefully reading and suggesting many ameliorations leading to the present paper.

\section{The plurifine Topology}\label{sec1}

The plurifine topology $\calf=\calf(\Omega)$ on a Euclidean open set $\Omega\subset\CC^n$ is the smallest topology that makes all plurisubharmonic functions on $\Omega$, denoted by $\PSH(\Omega)$, continuous. Since plurisubharmonic functions are already upper semi-continuous, a \emph{local subbasis} at $a\in\Omega$ is given by the sets
\begin{equation}
\label{eq2:1}
U(a,B,f)=\{z\in B: f(z)>0\},
\end{equation}
where $B\subset\Omega$ is a ball about $a$, and $f\in\PSH(B)$ with $f(a)>0$.
It is easy to see that if $\Omega_1\subset\Omega$ are open in $\CC^n$, then $\calf(\Omega)$ induces on $\Omega_1$ the topology $\calf(\Omega_1)$. Similarly,  $\calf(\CC^n)$ induces on a complex affine hyperplane $H\cong \CC^k$ the topology  $\calf(\CC^k)$.

\begin{prop}\label{prop2:1}
 The sets $U(a,B, f)$ form a local \emph{basis} for $\calf$.
\end{prop}
 This result was first observed in \cite{BT2}. See \cite{EW1, EW2} for a proof.

We gather some properties of the plurifine topology, cf.~\cite[Chapter 11]{D}.
\begin{theorem}[Properties of the plurifine topology $\calf$]\hfill
\begin{enumerate}
\item $\calf$ is \emph{quasi-Lindel\"of}, that is, every arbitrary union of $\calf$-open sets is the union of a countable subunion and a pluripolar set.
\item $\calf$ is \emph{completely regular}, that is, for every $\calf$-closed set $A\subset\Omega$ and $a\in \complement A$ there exists an $\calf$-continuous $f$ such that $f|A=0$ and $f(a)\ne 0$.
\item $\calf$ is \emph{Baire}, that is, a countable intersection of sets that are open and dense in $\Omega$, is dense in $\Omega$ (relative to $\calf$).
\end{enumerate}
\end{theorem}

It is a well-known fact that the fine topology on $\RR^n$ is connected, cf.~\cite{F1}, but it was for quite some time an open question whether $\calf$ is locally connected, cf.~\cite{F6}. The first proof of this result was given in \cite{EW1}. In \cite{EW2} another proof gives as a by-product some extra information about the relation between plurifine topology in an open set $\Omega$ and the fine topology on complex lines in $\Omega$. It is based on the solutions by Nevanlinna, \cite{N} and Beurling \cite{B} to the Carleman-Milloux problem.

Recall the \emph{extremal function} (or \emph{harmonic measure}) of a set $E\subset\overline{D}$, where $D$ is open in $\CC$,
\[\omega(z,E,D)=\sup\{f(z):\ f\in\SH_-(D)\ \text{and}\ \limsup_{D\ni w\to E} f(w)\le -1\},
\]
where $\SH_-(D)$ denotes the the space of negative subharmonic functions on $D$.
 Note that $\omega$ need not be subharmonic (as function of $z$), but its upper semicontinuous regularization,
 \[\omega^*(z,E,D)=\limsup_{D\ni w \to z} \omega(z,E, D),
 \]
will be subharmonic. Notice that $\omega\le 0$; classical harmonic measure equals $-\omega$.

\begin{theorem}[\cite{N,B}] \label{thm2.2}Let $F$ be compact in the unit disc $\DD$, and let $\tilde F=\{r: \exists \theta: re^{i\theta}\in F\}$ be its circular projection.
Then for $z\in \DD\setminus F$ we have
\[\omega(z,F,\DD)\le\omega(-|z|,\tilde F,\DD).
\]
\end{theorem}

A proof may also be found in \cite{Ransf}. One needs two lemmas.

\begin{lemma}[\cite{EW2}]\label{lemma2.4}$C^1$-arcs are connected in the fine topology on $\CC$.
 \end{lemma}
 \begin{proof}
We can take $I=[0,1]\subset\RR$.  Let $J:[0,1]\to\CC$ be a $C^1$-arc. Suppose that $J=E_1\cup E_2$, where $E_j=U_j\cap J$ for some f-open set $U_j\subset\CC$, ($j=1,2$), and $E_1\cap E_2= \emptyset$. Let $F_j=J^{-1}(E_j)$, and let $x\in F_j$. Then $F_{1-j}$ is thin at $x$. (cf.~\cite{Ransf}). From Wiener's criterion one easily gets
\begin{equation}\label{eq2.4a}
\lim_{t\to 0}\frac{\Capa(F_{1-j})\cap [x-t,x+t]}{t}=0.\end{equation}
Let $l$ denote Lebesgue measure on $[0,1]$. Because $l(F)<4\Capa(F)$, \eqref{eq2.4a} remains valid with $\Capa$ replaced by $l$. It follows that the function
\[f(s):=\int_0^s{\bf 1}_{F_1}(t)\, dt\]
is differentiable on $[0,1]$ and $f'|F_1\equiv 1$, $f'|F_2=0$. The intermediate value theorem for differentiable functions implies that either $F_1$ or $F_2$ must be empty.
\end{proof}

Fuglede proved the following.
\begin{lemma}[\cite{Fu75}]\label{lemma2.4a} If two point $p,q\in\CC$ in an $f$-open set  $U$ are joined by a (Euclidean) continuum $K\subset U$, then there exists a polygonal path in $U$ joining $p$ and $q$.
\end{lemma}

All this gives local connectedness of the fine and the plurifine  topology.

\begin{theorem}[\cite{EW2}]\label{thm2.3}
Let $d<c<0$ and $0<r<1$.
\begin{enumerate}
\item[a.] There exists $k>0$ such that for every $\varphi\in\SH_-(\DD)$ with $\varphi(0)>c$ and for every point $a$ in the ($f$-open) set
\[V=\{\varphi>c\}\cap \{|z|<r\}
\]
there exists a circle $C(a,\delta_{\varphi,a})$ with $\delta_{\varphi,a}>k$, which is contained in $W=\{\varphi> d\}$.
\item[b.] Moreover the set $\tilde W=W\cap \overline{B(a,\delta_{\varphi,a})}$ is polygonally connected, and therefore $f$-connected.
\end{enumerate}
\end{theorem}
\begin{proof}[Sketch of proof]
 a.) This follows from Theorem \ref{thm2.2}. For b.) one has to observe that $a$ and $C(a,\delta_{\phi,a})$ must belong to the same component of $W$, because of the maximum principle. Then Lemma \ref{lemma2.4a} and Lemma \ref{lemma2.4}  apply.

\end{proof}

For the plurifine topology this yields
 \begin{theorem}\label{thm2.5}  The plurifine topology $\calf$ is weakly locally connected, hence locally connected.

\smallskip
 More precisely, suppose that $U=U(a, B(a,R_0) ,f)$ is a basic neighborhood in $\calf$. Let $R<R_0$,  $c<f(a)$, and  $V=\{f\ge f(a)\}\cap\{\|z-a\|<R\}$. Then there exists a constant $r>0$ such that for every complex line $L$ passing through $v\in V$ the set $\{f>c\}\cap L$ contains a circle $C(v,r_{v,L})$ with $r_{v,L}\ge r$, and the set $\{f>c\}\cap L\cap \overline{B(v,r_{v,L})}$ is polygonally connected.

 Now let $X_{v,L}$ denote the $\calf$-component of $v$ in $L\cap U$. Then $\cup_{v\in L}X_{v,L}$ is an $\calf$-connected set and contains the $\calf$-neighborhood $\{f>c\}\cap B(v,r)$of $v$.

\end{theorem}

\section{$\calf$-pluri\-sub\-har\-monic functions}\label{sec2}
Let $\SH(D)$ denote the subharmonic functions on a domain $D\subset\RR^n$, $\SH_-(D)$ the negative subharmonic functions on $D$, $\PSH(\Omega)$ the plurisubharmonic functions on a domain $\Omega\subset\CC^n$, and $\PSH_-(\Omega)$ the negative plurisubharmonic functions on $\Omega$.
We start by defining finely subharmonic functions.

\subsection{Finely subharmonic functions}\label{ss2.1}
Let $D$ be a bounded domain in $\RR^n$ and let $E\subset D$. For $u\in\SH_-(D)$ the reduced function $R^E_u$ is defined by
\[R^E_u(z)=\sup\{v(z): v\in\SH_-(D), \ v\le u \text{\ on\ } E\}.
\]
The upper semicontinuous regularization of $R^E_u$ is denoted by $\check{R}^E_u$. It is a subharmonic function on $D$. Next let $z\in D$ and $\delta_x$ the point mass at $x$. The \emph{sweep} or \emph{swept out}  of $\delta_z$ with respect to $E$ is the unique Radon measure $\delta^E_z$ defined by
\[\check R^E_u(z)=\int u d\delta^E_z.
\]
The key observation is that $u\mapsto \check R^E_u(z)$ extends as a continuous linear functional on $C(E)$, and the Riesz Representation Theorem applies. In case $F$ is a subset of the boundary $E$ of a domain $ D$ and $z\in D$, one sees that $\delta^E_z(F)=-\omega(z,F,D)$ the harmonic measure of $E$ relative to $z\in D$.

\begin{definition}\label{def21.1}
A function $f$ defined on a fine open set $U\subset\RR^n$ is called \emph{finely subharmonic} if
\begin{enumerate}
\item $f$ is finely upper semicontinuous.
\item
\[f(z)\le \int_{\partial_f V} f d\delta^{\comp V}_z.
\] for $V$ in some local base of the fine topology at $z$.
\item $f\not\equiv -\infty$ on every fine component of $U$.
\end{enumerate}
\end{definition}

Clearly, the restriction of a subharmonic function to an $f$-domain is finely subharmonic.
A bounded finely subharmonic function on a Euclidean domain is subharmonic. In $\RR^2$ the boundedness may be dropped. In $\RR^n$, $n\ge 3$, there are examples of $f$-subharmonic functions on Euclidean domains that are not subharmonic, cf.~\cite{F1,F2}.

The finely subharmonic functions on a fine domain $U$ will be denoted by $\fSH(U)$.

Finely subharmonic functions share many properties with ordinary subharmonic functions, e.g. we have the following result, cf.~\cite[Lemma 10.1]{F1}.
\begin{prop}
Let $V\subset U$ be fine open sets, $u\in \fSH(U)$, and $v\in \fSH(V)$.
Suppose that $\flimsup_{V\ni z\to w} v(z)\le u(w)$ for all $w\in\partial_f V$, then the function
\begin{equation*}
\Psi(z)=\begin{cases} u(z)&\text{if $z\in U\setminus V$}\\
										\max(u(z),v(z)) & \text{ if $z\in V$}
				\end{cases}
\end{equation*}
belongs to $\fSH(U)$.
\end{prop}

\subsection{First properties of finely plurisubharmonic functions}\label{ss2}
For a compact set $K\subset\CC^n$,
let $S(K)$ denote the uniform limits on $K$ of bounded continuous plurisubharmonic functions defined in (shrinking) neighborhoods of $K$.

\begin{definition}[Plurifinely plurisubharmonic function]\label{def2.2} Let $\Omega$ denote an $\calf$-open
(i.e., pluri\-finely open) subset of $\CC^n$.

(i) A function
$f:\Omega\to[-\infty,+\infty[\,$ is said to be {\it weakly $\Cal
F$-pluri\-sub\-har\-monic} if $f$ is $\calf$-upper
semi\-con\-tinu\-ous and, for every complex line $L$ in $\CC^n$,
the restriction of $f$ to any $\calf$-component of the finely open subset $L\cap\Omega$ of $L$
is either {\it finely subharmonic} or $\equiv-\infty$.

(ii) A function $f:\Omega\to\RR$ is said to be {\it $\calf$-cpsh}
if every point of \,$\Omega$ has a compact $\calf$-neighborhood $K$
in $\Omega$ such that $f|K\in{S(K)}$.

(iii) A function $f:\Omega\to[-\infty,+\infty[\,$ is said to be {\it
strongly $\calf$-pluri\-sub\-har\-monic} if $f$ is the pointwise
limit of a decreasing net of \,$\calf$-cpsh
functions on $\Omega$.
\end{definition}

Weakly $\calf$-plurisubharmonic functions were defined in \cite[Section 5]{ElKa},
 \cite[Definition 5.1]{EW2}).

 The concepts strongly and weakly $\calf$-pluri\-sub\-har\-monic are both $\Cal
F$-local ones (that is, these have the sheaf property: if a function is locally in $\fSH$ or in $\FPSH$, it is so globally).

One can fairly easily show that strongly $\calf$-plurisubharmonic functions are weakly $\calf$-plurisubharmonic. In case $n=1$ the notions are the same. A proof is indicated in Remark \ref{rem22.1}. Denote the class of all  weakly $\calf$-plurisubharmonic functions on an $\calf$-open set $\Omega$ by $\FPSH(\Omega)$. Then $\FPSH(\Omega)$ is a convex cone which is stable under taking the pointwise supremum of
finite families. Furthermore, $\FPSH(\Omega$) is stable under taking the
pointwise infimum for lower directed (possibly infinite) families, and is
closed under $\calf$-locally uniform convergence in view of analogous results for finely subharmonic functions, cf.~\cite[Lemma 9.6]{F1}.
Pointwise suprema of families of weakly $\Cal F$-pluri\-sub\-har\-monic functions are discussed in Theorem \ref{th3.9} below. The restriction of $f\in\FPSH(\Omega)$ to a complex affine subspace is of course weakly $\calf$-plurisubharmonic.

In the following two theorems we collect some further properties of weakly
$\calf$- pluri\-sub\-har\-monic functions.

\begin{theorem}\label{th2.3} $($\cite{EW2}$)$ Let $f$ be a weakly $\Cal
F$-pluri\-sub\-har\-monic function on an $\calf$-domain
$\Omega\subset\CC^n$ and let $E=\{z\in\Omega:f(z)=-\infty\}$.

{\rm(a)} If \,$f\not\equiv-\infty$ then $E$
has no $\calf$-interior point.

{\rm(b)} If \,$f\not\equiv-\infty$ then, for any $\calf$-closed set
$F\subset E$, $\Omega\setminus F$ is an
$\calf$-domain.

{\rm(c)} If \,$f\le0$ then either $f<0$ or $f\equiv0$.
 \end{theorem}

This result was known for the case $n=1$, cf.~\cite{F1}. The case $n>1$ is proven using this case and Theorem \ref{thm2.5}, cf.~\cite{EW2}.

The next theorem gives a handle on the local structure of weakly $\calf$-plurisubharmonic functions. It is of fundamental importance in plurifine pluripotential theory. Therefore we will provide its proof here.

\begin{theorem}[\cite{EW3}] \label{thm2.4} Let $f$ be a weakly $\Cal
F$-pluri\-sub\-har\-monic function on an $\calf$-open set
$\Omega\subset\CC^n$, that maps $\Omega$ into a fixed bounded interval
$]a,b[$.

Every point $z_0\in\Omega$ has an
$\calf$-open $\calf$-neighborhood $O\subset\Omega$ on which $f$ can
be represented as the difference $f=\phi_1-\phi_2$ between two bounded
plurisubharmonic functions $\phi_1$ and $\phi_2$ defined on some open
ball $B(z_0,r)$ containing $O$. Moreover, $r$, $O$, and $\phi_2$ will depend on $\Omega$ and $]a,b[$, but can be chosen independently of $f$.
\end{theorem}

We will follow \cite{EW3}.

\begin{proof}   We
may assume that $-1<f<0$ by scaling $f$ if necessary.
Let $V \subset \Omega$ be a compact $\Cal{F}$-neighborhood of
$z_0$. Since the complement $\complement V$ of $V$ is pluri-thin at
$z_0$, there exist $ 0<r<1$ and a pluri\-sub\-har\-monic function
$\varphi$ on $B(z_0,r)$ such that
 $$
\limsup_{z \to z_{0},\,z\in \complement V }\varphi(z)<\varphi(z_{0}).
 $$
Without loss of generality we may suppose that $\varphi $ is negative
on $B(z_{0}, r)$ and
$$
\varphi(z)=-1\;\text{on}\;B(z_0,r)\setminus
V\quad\text{and}\;\phi(z_0)=-1/2.
 $$
Hence
 \begin{equation}
f(z)+\lambda\phi(z)\le-\lambda\qquad\text{for }z\in \Omega\cap
B(z_0,r)\setminus V\text{ and }\lambda>0.\label{eq2.3}
 \end{equation}
Now define a function $u_{\lambda}$ on $B(z_0,r)$ by
 \begin{equation}
u_\lambda(z)=\begin{cases}\max\{-\lambda,f(z)+\lambda\phi(z)\}&\text{for
$z\in \Omega\cap B(z_0,r)$}\\
-\lambda &\text{for $z\in B(z_0,
r)\setminus V$}.
 \end{cases}\label{eq2.4}
 \end{equation}
This definition makes sense because
$\bigl(\Omega\cap B(z_0,r)\bigr)\bigcup\bigl(B(z_0, r)\setminus
V\bigr)=B(z_0, r)$, and the two definitions agree on $\Omega\cap B(z_{0},
r)\setminus V$ in view of \eqref{eq2.3}.

Clearly, $u_{\lambda}$ is weakly $\calf$-pluri\-sub\-har\-monic on
$\Omega \cap B(z_0, r)$ and on $B(z_0, r)\setminus V$, hence on all of
$B(z_0, r)$ in view of the sheaf property, cf.\ \cite{EW2}. Since
$u_{\lambda}$ is bounded on $B(z_0, r)$, it follows from \cite[Theorem 9.8]{F1}
that $u_{\lambda}$ is subharmonic on each complex line
where it is defined. It is well known that a bounded function, which
is subharmonic on each complex line where it is defined, is
pluri\-sub\-har\-monic, cf. \cite{Le1} or \cite[p.\
24]{Le2}. Thus, $u_{\lambda}$ is pluri\-sub\-har\-monic on $B(z_0,r)$.

Since $\phi(z_0)=-1/2$, the set $O=\{z\in\Omega:\phi(z)>-3/4\}$ is an
$\Cal{F}$-neighborhood of $z_0$, and because $\phi=-1$ on $B(z_0,
r)\setminus V$ it is clear that $O\subset V\subset\Omega$.

Observe now that $-4\leq f(z)+ 4\phi(z)$ for every $z\in O$. Hence
$f=\phi_1-\phi_2$ on $O$, with $\phi_1=u_4$ and $\phi_2=4\phi$, both
pluri\-sub\-har\-monic on $B(z_0,r)$.

\end{proof}

\begin{corollary}
Every weakly $\calf$-plurisubharmonic function $f$ on $\Omega$ is $\calf$-continuous.
Hence, if for some $z\in\Omega$ $f(z)>-\infty$, then there is an $\calf$-neighborhood $O$ where $f$ can be written as a difference of plurisubharmonic functions defined in a neighborhood of $O$.
\end{corollary}
For an unbounded $f$ just note that, given $d<c$, the set $E=\{f<c\}$ is $\calf$-open, and on $E$ the function $\max\{f, d\}$ is $\calf$-continuous, hence $\{d<f<c\}$ is $\calf$-open.
\begin{remark} \label{rem22.1}
The  \emph{extremal function} of $B\setminus V$, i.e. the function $\Psi^*=(\sup\{h: h\in \PSH_-(B), h|(B\setminus V)\le -1\})^*$ could have been used in the proof of Theorem \ref{thm2.4} instead of $\phi$. Then one should take $\phi_2=4\Psi^*$.
In case $n=1$ this function is harmonic except on the boundary of $V$. Approximating $B\setminus V$ from the inside with compact sets $K_n\nearrow B\setminus V$, and forming the corresponding $\Psi_n^*$, we have $\Psi_n^*$ harmonic in a neighborhood of $z_0$ and  $\Psi^*_n\downarrow \Psi^*$. The Brelot property, cf.~\cite{Fu72a} states that on a suitable compact fine neighborhood $K$ of $z_0$, both $\phi_1$ and $\phi_2$ are continuous in the Euclidean topology. Then $f=\lim\phi_1-\Psi^*_n$, a uniform limit on $K$ of subharmonic functions defined in a neighborhood of $K$, as announced after Definition \ref{def2.2}.
In case $n\ge2$ this breaks down for two reasons. $\Psi^*$ will in general not be pluriharmonic on $B\setminus \overline V$, and there is no Brelot property for plurisubharmonic functions.
\end{remark}

\begin{theorem}\label{thm2.41} Suppose that $f$ is a weakly $\calf$-plurisubharmonic function on an $\calf$-domain $\Omega$.
 If $f\not\equiv-\infty$
then $E=\{z\in\Omega:f(z)=-\infty\}$ is an $\calf$-closed, pluripolar subset of \,$\CC^n$.
\end{theorem}

A few words about the proof. In an $\calf$-neighborhood of $z_0\in E$, $f$ will be negative. Now keep the notation of Theorem \ref{thm2.4} and write $f_n=\max(f,-n)/(4n)$ as $u_n-\Psi^*$ with $\Psi^*$ as in Remark \ref{rem22.1}, and $u_n$ defined completely analogous to \eqref{eq2.4}
\begin{equation}
u_n(z)=\begin{cases}\max\{-1,f_n(z)+\Psi^*(z)\}&\text{for
$z\in \Omega\cap B(z_0,r)$}\\
-1 &\text{for $z\in B(z_0,r)\setminus V$}.
 \end{cases}\label{eq2.5}
 \end{equation}
The $u_n$ are plurisubharmonic and will increase to $\Psi^*$ except at points of $E$,
which will imply that $E$ is pluripolar, first $\calf$-locally and then by the quasi-Lindel\"of property also globally.

\smallskip

It is unknown whether plurisubharmonic functions have the Brelot property, but  a weak version of it holds.

\begin{theorem}[quasi-Brelot property, cf.~\cite{EW2}]
Suppose that $f$ is a weakly $\calf$-plurisubharmonic function on an $\calf$-domain $\Omega$, then there exists a pluripolar $E\subset\Omega$ such that every $z\in\Omega\setminus E$ admits an $\calf$-neighborhood $K_z$ such that $f|K_z$ is Euclidean continuous.
\end{theorem}

Notice that $f$-subharmonic functions on Euclidean domains in $\RR^n$, ($n\ge 3$) need not be subharmonic, but there is no difference between $\calf$-plurisubharmonic and plurisubharmonic functions on Euclidean open sets, cf.~\cite{EFW11}.

\begin{prop}\label{prop2.14} Let $\Omega$ be a Euclidean open
subset of \,$\CC^n$. For a function
$f:\Omega\to[-\infty,+\infty[\,$ the following are equivalent:

\begin{description}
\item[\phantom{i}\rm i] $f$ is pluri\-sub\-har\-monic $($in the ordinary sense$)$.
\item[\rm ii] $f$ is weakly $\calf$-pluri\-sub\-har\-monic and not
identically $-\infty$ on any com\-po\-nent of \,$\Omega$.
 \end{description}
\end{prop}

\subsection{Lelong characterization}\label{Bounded}
When Lelong defined pluri\-sub\-harmonic functions, \cite{Lelong42}, he set out from functions that are locally bounded from above and have the property that their restrictions to complex lines are subharmonic. These are indeed upper semicontinuous, and hence plurisubharmonic.
We show that this also holds in the weakly (and therefore also in the strongly) $\cal F$ situation.
\begin{theorem}
Suppose that $f$ is an $\cal F$-locally bounded from above function on an $\cal F$-domain $D$ with the property that $f|D\cap L$ is finely subharmonic for every complex line $L$.
Then $f$ is $\cal F$-PSH on $D$.
\end{theorem}
\begin{proof}
We first observe that if we can prove the result for the $\cal F$-locally bounded functions $\max\{f, -n\}$ we are done, because then $f$ is  the limit of a decreasing sequence of $\cal F$-psh functions.
Now assume that $f$ is $\cal F$-locally bounded. By  copying the proof of Theorem \ref{thm2.4}, we see that every point $z\in D$ admits a ball $B(z, r)$ and an $\cal F$-neighborhood $K_z$ on which $f=\psi-\phi$, where $\phi\in \PSH(B(z,r))$, and $\psi$ is defined on $B(z,r)$ and is again of the form $\max\{f+\psi, C\}$ glued to the constant $C$.  We observe that $\psi$ is a bounded function on $B(z,r)$ and its restriction to complex lines is subharmonic. Therefore, by Lelong's theorem $\psi\in PSH(B(z,r))$, hence $f$ is $\cal F$ continuous on $K_z$, and by varying $z$ and the sheaf property, also on $D$. It follows that $f$ is $\cal F$ plurisubharmonic.
\end{proof}

\section{$\Cal F$-pluri\-sub\-har\-monic functions as invariant f-subharmonic functions}
\subsection{Main Theorem}
 As is well-known, a plurisubharmonic function $f$ on a domain
$\Omega\subset\CC^n$ is subharmonic when considered as a function
on $\Omega\subset\RR^{2n}$, because the average of $f$ over a
sphere can be expressed in terms of the average of $f$ over the
circles that are intersection of the sphere with complex lines passing
through the center. While this approach does not work in the fine
setting, an analogous result nevertheless is valid.  Indeed, a
well-known characterization of pluri\-sub\-har\-monic functions (see
\cite[Th\'eor\`eme 1 (p.\ 18)]{Le2} or \cite[Theorem
2.9.12]{K}) may be adapted as follows.

\begin{theorem}\label{th3.1}  \cite{EFW11} Let $\Omega$ be $\calf$-open in $\CC^n$. A function $f:\Omega\to[-\infty,+\infty[$
 is weakly $\Cal
F$-pluri\-sub\-har\-monic if and only if $f$ is $\calf$-locally bounded from above
and for every $\CC$-affine bijection $h$ of \,$\CC^n$ the function $f\circ h$ is $\Bbb
R^{2n}$-finely subharmonic on each fine component of the $\calf$-open set $h^{-1}(\Omega)$ on which $f\circ h\not \equiv -\infty$.
 \end{theorem}

The prefix `$\RR^{2n}$-fine' refers to concepts relative to the
Cartan fine topology on $\CC^n\cong\RR^{2n}$. Recall that this
topology is finer than the plurifine topology $\calf$, \cite{F6}. This explains why the condition ``$\calf$-locally bounded'' occurs in the statement of the theorem.

For the proof of Theorem \ref{th3.1} one needs the following

\begin{lemma}\label{lemma3.2} \cite{EFW11} Let $u_1,u_2$ be bounded subharmonic functions
on an open set $B\subset\Bbb
R^n$, and consider the function $f=u_1-u_2$ on $B$. Let $U$ be a
finely open Borel subset of $B$. Then $f|U$ is finely subharmonic if and
only if the signed Riesz measure $\Delta f$ on $B$ has a positive
restriction to $U$.
 \end{lemma}

\begin{proof}[Indication of the proof of the \lq only if part' of Theorem \ref{th3.1}]

Writing $f$ \hfill\break\noindent $\calf$-locally as a difference of plurisubharmonic functions on an $\calf$-open set $U\subset\Omega$, we know that the restriction to a complex line $L$ is f-subharmonic, hence by Lemma \ref{lemma3.2} it has positive Riesz mass on $L\cap U$. Then a careful application of the definition of Riesz mass in distribution sense and Fubini's theorem lead to positivity of the Riesz mass on $U$. Another application of the lemma gives that $f$ is f-subharmonic on $U$. We can do so in an $\calf$-neighborhood of any point in $\Omega$. The sheaf property ensures that $f$ is f-subharmonic on $\Omega$.
\end{proof}

Lemma \ref{lemma3.3} and Lemma \ref{lemma3.5} are ingredients of the proof of the 'if part' of Theorem \ref{th3.1}.

\begin{lemma}\label{lemma3.3} Let $f$ be a bounded $\RR^{2n}$-finely subharmonic function
on an $\calf$-open set $\Omega\subset\CC^n$ and suppose that for
every $\CC$-affine bijection $h$ of $\CC^n$ the function $f\circ
h$ is $\RR^{2n}$-finely subharmonic on $h^{-1}(\Omega)$. Then every $z_0\in
\Omega$ admits a ({\em{compact}}) $\calf$-neighborhood $K_{z_0}$ such that
$f$ can be written as
 $$
f=f_1-f_2\quad \text{on\ } K_{z_0},
 $$
where $f_1, f_2$ are plurisubharmonic functions defined on a ball
$B(z_0,r)\supset K_{z_0}$.
\end{lemma}

The proof is similar to the proof of  Theorem \ref{thm2.4},  cf.~\cite{EFW11}. One employs the fact that $f_2$ is plurisubharmonic, and $f_1$ is now subharmonic on a Euclidean ball and remains so after affine transformation.

The next lemma is a consequence of results of Bedford and Taylor on slicing of currents, cf.~\cite{BT3}.
In $\CC^n$ we will write $z=(z_1,z_2,\ldots,z_n)=(z_1,z')$; similarly $0=(0,0')$ and, abusing notation, $\eps'^2=\prod_{j=2}^n\eps_j^2$, whereas $|z'|<\eps'$ stands for $|z_j|<\eps_j$, $j=2,\ldots,n$.

\begin{lemma}\label{lemma3.5} Let $w$ and $u$ be bounded
pluri\-sub\-har\-monic functions on a bounded domain $D\subset\Bbb
C^n$, and  let $\psi=\psi(z_1)$ be in $C^\infty_0$ on
$\{z\in D: z'=0'\}$. Then
\begin{equation}\label{eq3.5}
\begin{split}
&\int_{\{z_2=0,\ldots,z_n=0\}}\psi(z_1)w(z_1,0')\, dd^c u(z_1,0')\\
=
&\lim_{\epsilon'\downarrow 0}\frac{1}{2^{n-1}{\epsilon'}^2}
\int_{\{|z'|<\epsilon'\}}\psi(z_1)w(z)
dd^c|z_2|^2\wedge\ldots\wedge dd^c|z_n|^2 \wedge d d^cu.
\end{split}
\end{equation}
\end{lemma}

\begin{proof}[Indication of proof of the `if part' of Theorem \ref{th3.1}]
One easily reduces the proof of the `if part' of Theorem \ref{th3.1} to the
case where $f$ is bounded. With the notation
from Lemmas \ref{lemma3.3} and \ref{lemma3.5}, one first
observes that $dd^c f$ makes sense as a (1,1)-form because of Lemma \ref{lemma3.3}, and is $\ge 0$ on the
compact neighborhood $K=K_{z_0}$ of $z_0$ provided by Lemma \ref{lemma3.3}, because of multiple application of Lemma \ref{lemma3.2}. Next,
application of Lemma \ref{lemma3.5} shows that the restriction of $f$ to any complex
line passing through $z_0$ is finely subharmonic on a fine
neighborhood of $z_0$.
\end{proof}

From Theorem \ref{th3.1} we derive the following two results, one about
removable singularities for weakly $\calf$-pluri\-sub\-har\-monic
functions, and the other about the supremum of a family of such
functions.

\subsection{Extension over Polar Sets}\label{S-polar}

Let $\Omega$ be a domain in $\CC^n$ and let $E$ be a closed polar (with respect to $\RR^{2n}$) subset of $\Omega$. A theorem of Lelong states that if $f$ is a bounded plurisubharmonic function on $\Omega\setminus E$, then $f$ extends to a plurisubharmonic function on all of $\Omega$. The following theorem combines Lelong's idea and Theorem \ref{th3.1}.

\begin{theorem}
Let $U$ be an $\calf$-open set in $\CC^n$ and let $E$ be a subset of $U$ that is finely closed and (finely) polar. Suppose that $f$ is a bounded $\FPSH$ function on $U\setminus E$. Then there exists a function $g\in\FPSH(U)$ with $f=g|U\setminus E$.
\end{theorem}
\begin{proof} By Theorem \ref{th3.1} the function $f$ is finely subharmonic on $U\setminus E$,
hence by \cite[Theorem 9.14]{F1}, the function
\begin{equation}
g(z)=\begin{cases}
				f(z)& \text{if $z\in U\setminus E$},\\
				\flimsup\limits_{\substack{{w\to z}\\{w\in U\setminus E}}} f(w)& \text{if $z\in E$};
				\end{cases}
\label{eq:}
\end{equation}
is finely subharmonic. For $f\circ h$, where $h: \CC^n\to\CC^n$ is a complex affine map, the same holds on $h^{-1}(U)$, because polarity is preserved under affine maps. Now Theorem \ref{th3.1} in the reverse direction applies, and states that $g$ is $\calf$-PSH.
\end{proof}

There is a similar result about removable
singularities for weakly $\calf$-holo\-morph\-ic functions (for a definition see Definition \ref{def-weaklyhol} below):

\begin{corollary}\label{cor3.8} Let $h:\Omega\to\CC$ be $\calf$-locally
bounded on $\Omega$ $(\calf$-open in $\CC^n)$. If $h$ is weakly
$\calf$-holo\-morph\-ic on $\Omega\setminus E$ $(E$ finely closed
and $\RR^{2n}$-polar in $\CC^n)$ then $h$ extends uniquely to a weakly
$\calf$-holo\-morph\-ic function $h^*:\Omega\to\CC$, given by
 $$
h^*(z)
=\calf\text{-}\lim_{\substack{\zeta\to z\\ \zeta\in \Omega\setminus E}}
h(\zeta),\qquad z\in\Omega.
 $$
 \end{corollary}

\subsection{Behavior of families of $\calf$-plurisubharmonic functions}

\begin{theorem}\label{th3.9} Let $\Omega$ denote an $\calf$-open subset of
$\CC^n$.  For any family of weakly
$\calf$-pluri\-sub\-har\-monic functions $f_\alpha$ on $\Omega$, such that $f:=\sup_\alpha f_\alpha$ is $\calf$-locally
bounded from above, the
least $\calf$-upper semi\-con\-tinu\-ous majorant $f^*$ of the
pointwise supremum $f$ is likewise weakly $\Cal
F$-pluri\-sub\-har\-monic on $\Omega$, and
$\{z\in\Omega:f(z)<f^*(z)\}$ is pluripolar.
 \end{theorem}

In case  $\Omega$ is {\it Euclidean} open, we find

\begin{corollary}\label{cor3.10} For any family $\{f_\alpha\}$ of ordinary
pluri\-sub\-har\-monic functions on a Euclidean open set
$\Omega\subset\CC^n$ such that $f:=\sup_\alpha f_\alpha$ is locally
bounded from above, the least pluri\-sub\-har\-monic majorant of $f$
exists and can be expressed as the upper semicontinuous regularization
of $f$ in the Euclidean topology on $\CC^n$, as well as in the $\Cal
F$-topology and in the $\RR^{2n}$-fine topology; that is, $\bar
f=f^*=\check f$.
 \end{corollary}

The version of this involving the {\it Euclidean} topology is due to
Lelong \cite{Le1}, or see \cite[p.\ 26]{Le2} or \cite[Theorem
2.9.10]{K}.

For the proof of the theorem and its corollary we refer to \cite{EFW11}.

\section{Biholomorphic invariance}

Notions in complex analysis should remain invariant under holomorphic change of coordinates. This is indeed the case for weakly $\calf$-plurisubharmonic (and weakly $\calf$-holomorphic) functions. But here we do want a bit more, namely that the composition of such a function with a weakly plurifinely holomorphic map is again of the same category.

We recall the relevant notions.

\begin{definition}[finely holomorphic function, \cite{F3,F7}]
 Let $\Omega$ be a fine domain in $\CC$. A function $f:\Omega\to \CC$ is called finely holomorphic if for every $z\in \Omega$ there exists a (compact) fine neighborhood $K_z$ of $z$ and a smooth function $\phi$ defined on a Euclidean neighborhood of $K_z$ such that $\phi=f$ on $K_z$ and $\bar \partial \phi=0$ on $K_z$.
\end{definition}

\begin{definition}\label{def-weaklyhol}
 An $\calf$-continuous function $f$ on an $\calf$-domain $\Omega$ in $\CC^n$ is called
weakly $\calf$-holomorphic if its restriction to $\Omega\cap L$ is finely holomorphic for every complex line $L$ that meets $\Omega$.

It is called strongly $\calf$-holomorphic if for every $z\in \Omega$ there exists a (compact) $\calf$-neighborhood $K_z$ on which it is the uniform limit of holomorphic functions defined on Euclidean neighborhoods of $K_z$.
\end{definition}

In fact, we could replace $\calf$-continuous in the definition of weak holomorphy by $\calf$-locally bounded: the real and imaginary part of $f$ would be weakly $\calf$-pluri\-sub\-harmonic, hence $\calf$-continuous. In the case $n=1$ it
 is known that weak and strong fine holomorphy is the same, cf.~\cite{F3}.

\begin{definition}[Plurifinely biholomorphic map, \cite{EFW11}]
\label{def4.3} A {\it strongly $\calf$-biholo\-morphic
map} $h$ from an $\calf$-open set $U\subset\CC^n$ onto its image in
$\CC^n$ is an $\calf$-homeo\-morph\-ism with the property that there
exists for every $z\in U$ a compact $\calf$-neighborhood $K_z$ of $z$
in $U$ and a $C^\infty$-diffeo\-morph\-ism $\Phi_z$ from an open
neighborhood of $K_z$ to its image in $\CC^n$ such that
$\Phi_z|_{K_z}=h|_{K_z}$ and that $\Phi_z|_{K_z}$ is a $C^2$-limit of
biholo\-morphic maps defined on open sets containing $K_z$.
 \end{definition}

Finely holomorphic functions of one variable are in fact locally strongly finely biholomorphic at points were they are locally injective. They can be approximated $\calf$-locally uniformly by holomorphic functions at any point of their domain.

\begin{definition}
We call an $n$-tuple $(h_1,\ldots,h_n)$ of
strongly/weakly $\calf$-holo\-morph\-ic functions $h_j:U\to\CC$
defined on some $\calf$-open $U\subset\CC^m$) a {\it
strongly/weakly} $\calf$-{\it holo\-morph\-ic map} (or {\it curve} if
$m=1$).
\end{definition}

We now have

\begin{theorem}[\cite{EFW11}]\label{th4.6} Let $h:U\to\Omega$ be a
weakly $\calf$-holo\-morphic map from an $\calf$-open $U\subset\CC^m$ into an $\calf$-open $\Omega\subset\CC^n$.
The composition $f\circ h$ of a weakly $\calf$-pluri\-sub\-har\-monic (resp.\ weakly $\calf$-holo\-morph\-ic)
function $f$ on  $\Omega$ with $h$ is weakly $\calf$-pluri\-sub\-har\-monic $($resp.\ weakly
$\calf$-holo\-morph\-ic$)$ on $U$.
\end{theorem}

As for a sketch of the proof, by a fairly easy change of coordinates one can give a proof if $f$ is holomorphic. To pass merely to strongly $\calf$-biholomorphic maps $f$ requires more effort. By the results of Section \ref{sec3} it is sufficient to show that $h\circ f$ is f-subharmonic. To employ the approximation property, one resorts to the description of f-subharmonic functions in terms of the Dirichlet spaces of Beppo Levi and Deny, cf.~\cite{DeLi}, that was studied in \cite{F4}.

The case of a weakly $\calf$-holomorphic map $f$ reduces to that of a weakly $\calf$ holomorphic curve. Such curves are locally injective except for a countable set of points. At a point where the curve is injective, it is the restriction to a complex line of a strongly $\calf$-biholomorphic map.

We refer to \cite{EFW11} for details.

\section{Applications to pluripolar hulls}\label{sec3}

In this section we will review some results concerning pluripolar hulls of graphs.

\begin{definition}\label{def3.1}
Let $E$ be a pluripolar subset of an open set $\Omega\subset\CC^n$.
The \emph{pluripolar hull} $E^*_\Omega$ of $E$ with respect to $\Omega$ is the set
\[
E^*_\Omega=\{z\in\Omega: \forall h\in\PSH(\Omega) \text{\ if $ h|E=-\infty$ then $h(z)=-\infty$}\}.
\]
We will write $E^*$ for $E^*_{\CC^n}$.
\end{definition}

The notion was introduced by Zeriahi in \cite{Ze89}. In case $E$ is an analytic variety, in particular if $E$ is the graph of a holomorphic function on a domain in $\CC$, interesting results were obtained.

Sadullaev, \cite{Sa}, and Levenberg, Martin and Poletsky, \cite{LMP} showed that for certain holomorphic functions defined by a lacunary series on the unit disc in $\CC$, the graph $\Gamma_f$ equals $\Gamma_f^*$. Answering questions of Sadullaev, cf.~\cite{Sa}, Levenberg and Poletsky \cite{LP99} showed that  if $\alpha\in\RR\setminus \QQ$ then $\Gamma_{z^\alpha}^*=\Gamma_{z^\alpha}$.
Here, abusing the notation, $\Gamma_{z^\alpha}=\{(z,w): |w|=|z|^\alpha, \arg w\in\{\alpha(\arg z+2k\pi): \ k\in\ZZ\}\}$, the graph of the complete analytic function $z^\alpha$. The author showed that if $f$ is a holomorphic function except for isolated singularities on a domain $\Omega\subset\CC$, then $(\Gamma_f)^*_\Omega=\Gamma_f$, \cite{W1,Wi00}. These results were in support of a conjecture by Levenberg, Martin and Poletsky, \cite{LMP} stating that if $E$ is an analytic set that admits no analytic extension, then  $E=E^*$. However, Edigarian and the author gave counterexamples, cf.~\cite{EdW1, EW03}, which were followed by many others, cf.~\cite{Zw05, Si03, PW}. Eventually Edlund and J\"oricke, \cite{EJ}, made the connection with fine holomorphy, observing that in all the available counterexamples the set $E$ under consideration admits no analytic extension, but it does admit so called \emph{fine analytic extension}. Their results were extended in \cite{EEW,EW1}.

\smallskip
For any set $E\subset\CC^m$, $m\in\Bbb N$, and any function
$h:E\to\CC$ we denote by $\Gamma_h(E)=\{(z,h(z)):z\in E\}$ the
graph of $h|E$ and by $\Gamma_h(E)^*_{\CC^{m+1}}$ the pluripolar
hull of $\Gamma_h(E)$.

\begin{prop}[\cite{EFW11}]\label{prop2.16} Let $h$ be a weakly $\calf$-holomorphic
function on an $\calf$-domain $U\subset\CC^m$.

{\rm(a)} If \,$h\not\equiv0$, the set $h^{-1}(0)$ of zeros of \,$h$ is
pluripolar in $\CC^m$. Also, the graph $\Gamma_h(U)$ of
$h$ is pluripolar in $\CC^{m+1}$.

{\rm(b)} If \,$E$ is a non-pluripolar subset of $U$ then $\Gamma_h(E)\subset \Gamma_h(U)$
is pluripolar, and $\Gamma_h(U)\subset\Gamma_h(E)^*_{\Bbb
C^{m+1}}$.
 \end{prop}

With $h$ supposed {\it strongly} $\calf$-holomorphic on $U$,
Proposition \ref{prop2.16} was obtained in \cite[Corollary 4.4 and Theorem 4.5]{EW3},
extending \cite[Theorem 6.4]{EW2}, and \cite[Theorem 3.5]{EEW}

\begin{proof}[Sketch of proof of Proposition \ref{prop2.16}] (a) The function $\log|h|$ is weakly $\Cal
F$-pluri\-sub\-har\-monic on $U$. Since $\log|h(z)|=-\infty$ for $z\in
h^{-1}(0)$, but $\log|h|\not\equiv-\infty$, it follows fromTheorem \ref{thm2.41} that
the set $h^{-1}(0)$ is pluripolar.

Apply this result to the function $(z,w)\mapsto w-h(z)$,
 which is weakly $\calf$-pluri\-sub\-har\-monic
and $\not\equiv-\infty$ on $U\times\CC$. Since $\log| w-h(z)|$ equals
$-\infty$ on $\Gamma_h(U)$ we conclude that $\Gamma_h(U)$ is
pluripolar.

(b) Now suppose that we have a 
pluri\-sub\-har\-monic function $f$ on $\CC^{m+1}$ such that
$g(z)=f(z,h(z))=-\infty$ for every $z\in E$. As $g$ is $\calf$-plurisubharmonic on $U$ by Theorem \ref{th4.6} and $E$ is not pluripolar, hence  by Theorem \ref{thm2.41} also not $\calf$-pluripolar in $U$,  it follows
that $f(z,h(z))=-\infty$  for  $z\in U$, and therefore
$\Gamma_h(U)\subset\Gamma_h(E)^*_{\CC^{m+1}}$.
\end{proof}

In \cite{EW04} the following result was proved.

\begin{theorem}\label{thm:5.9}
Let $D$ be an open set in $\CC$ and let $A$ be a closed polar subset of $D$.
Suppose that $f$  is holomorphic on $D\setminus A$ and that $z_0\in A$.
Assume that $U\subset\CC$ is an open set. Then $(\Gamma_f\cap (D\times U))^\ast_{D\times U}\subset (\Gamma_f\cap (D\times U))\cup (\CC\times A)$.
Moreover, the following conditions are equivalent:
\begin{enumerate}
\item\label{eq3:1} $(\{z_0\}\times\CC)\cap (\Gamma_f\cap (D\times U))^\ast_{D\times U}=\varnothing$;
\item\label{eq3:2} there exists a sequence of open sets $V_1\subset  V_2\subset\dots\Subset U$
such that $\cup_j V_j=U$ and the set $\{z\in D\setminus A: f(z)\in U \setminus{V_j}\}$
is not thin at $z_0$ for any $j\ge1$.
\item\label{eq3:3} for any open set $V\Subset U$ the set
$\{z\in D\setminus A: f(z)\in U\setminus {V}\}$ is not thin at $z_0$.
\end{enumerate}
Moreover, if the set $\{z\in D\setminus A: f(z)\not\in V\}$
is thin at $z_0$ for some open set $V\Subset U$, then there exists a $w_0\in\overline{V}$,
such that $(z_0,w_0)\in(\Gamma_f\cap D\times U)^\ast_{D\times U}$.
\end{theorem}

The formulation in the language of fine holomorphy is much more transparent:

\begin{theorem}[cf.~\cite{EEW}]\label{thm:5.9a}
Let $D$ be an open set in $\CC$ and let $A$ be a closed polar subset of $D$.
Suppose that $f$ is holomorphic on $D\setminus A$ and that $z_0\in A$.
Assume that $U\subset\CC$ is an open set. Then $(\Gamma_f\cap (D\times U))^\ast_{D\times U}\subset (\Gamma_f\cap (D\times U))\cup (\CC\times A)$.
Moreover, the following conditions are equivalent:
\begin{enumerate}
\item\label{eq3:1a} $(\{z_0\}\times\CC)\cap (\Gamma_f\cap (D\times U))^\ast_{D\times U}\ne \varnothing$;
\item\label{eq3:2a} $f$ admits a finely holomorphic continuation $\tilde f$ to a fine neighborhood of $z_0$ and $\tilde f(z_0)\in U$.
\end{enumerate}
Moreover, if this is the case, then
\[(z_0,\tilde f(z_0))=(\{z_0\}\times \CC)\cap(\Gamma_f\cap (D\times U))^\ast_{D\times U}.\]
\end{theorem}

Some parts are now easy to see: by Proposition \ref{prop2.16} we have  \eqref{eq3:2a} implies \eqref{eq3:1a}, and  also $(z_0,\tilde f(z_0))\in (\Gamma_f\cap (D\times U))^\ast_{D\times U}$. To show that $(\Gamma_f\cap (D\times U))^\ast_{D\times U}\subset (\Gamma_f\cap (D\times U))\cup (\CC\times A)$ remains difficult.

\subsection*{More pluripolar hulls of graphs}
Here we will shortly depict some of the examples mentioned in connection with non trivial pluripolar hulls. In all these examples there is a holomorphic function $f$ which admits no analytic continuation, but the graph $\Gamma_f$ has a non trivial pluripolar hull.

Suppose that $D_1\subset D_2$ are two domains such that $D_2\setminus D_1$ has a point of density $z_0\in D_2$. Then there exist a holomorphic function on $D_1$, which cannot be analytically extended so that $(\Gamma_f^*)_{D_2\times \CC} \ne \Gamma_f$. This was shown in \cite{EdW1}. In hindsight the function $f$ is a finely holomorphic function on a fine domain in $D_2$ that contains $D_1\cup\{z_0\}$.

Let $\DD$ be the unit disc. Siciak gave an example of an $f\in A^\infty(\DD)$ which admits not even a pseudocontinuation in the sense of Ross and Shapiro, but $\Gamma_f^*$ contains the graph of a meromorphic function on $|z|>1$, cf.~\cite{Si03}.

We consider Blaschke products $B$ on the unit disc $\DD$ and note that $B$ also defines a meromorphic function on $\CC\setminus \overline \DD$.  Zwonek, \cite{Zw05} constructed Blaschke products $B$  that do not admit analytic continuation, and have the property that the pluripolar hull of the graph $\Gamma_B$ contains the graph of $B$ over $\CC\setminus S$, where $S$ is the closure of the set of poles of $B$.  Multiplying such a $B$ with multiple valued holomorphic functions, he obtains examples of non extendable holomorphic functions on the disc with graphs having a pluripolar hull consisting of several sheets.

In the same vein is the example \cite{PW} where a Cantor type set $E$ is constructed and a non extendible holomorphic function $f$ on $\CC\setminus E$, with the property that $\Gamma_f^*$ contains two sheets over $\CC\setminus E$.

The point we wish to make is that in all these examples it can be seen that the function $f$ does admit a finely holomorphic continuation and the graph of the maximal finely holomorphic continuation is contained in the pluripolar hull. Sofar no other points in the pluripolar hull are found. This leads us to state a modified Levenberg Martin Poletsky conjecture.

\begin{conjecture} Suppose that $f$ is a finely holomorphic function on a fine domain $U$ in $\CC$. Then $\Gamma^*_f$ equals the graph of the maximal finely holomorphic continuation of $f$.
\end{conjecture}

It is here understood that such a maximal finely holomorphic continuation may be multiple valued. The theory developed in the present paper shows that $\Gamma^*_f$ contains the graph of the maximal finely holomorphic continuation of $f$. However, it does not clarify why equality would hold.

\bibliographystyle{amsplain}

\end{document}